\documentclass[reqno]{amsart}

\usepackage{amssymb}
\usepackage{graphicx}
\usepackage{amscd}
\usepackage[pagebackref]{hyperref}
\usepackage{color}
\usepackage{tabularx}
\usepackage[table]{xcolor}
\usepackage{float}
\usepackage{graphics,amsmath,amssymb}
\usepackage{amsthm}
\usepackage{amsfonts}
\usepackage{latexsym}
\usepackage{epsf}
\usepackage{xifthen}
\usepackage{mathrsfs}
\usepackage{dsfont}
\usepackage{makecell}
\usepackage{subfig}
\usepackage{amsmath}
\usepackage{listings}
\usepackage{etoolbox}
\usepackage{fancyhdr}
\usepackage{pdflscape}
\usepackage[title,toc,titletoc]{appendix}
\usepackage{enumitem}
\usepackage[noadjust]{cite}
\usepackage{tikz}
\usetikzlibrary{automata,positioning,arrows}
\usepackage{young}
\usepackage[object=vectorian]{pgfornament} %%  http://altermundus.com/pages/tkz/ornament/index.html
\usepackage{lipsum,tikz}
\usepackage{multirow}
\usepackage[OT2,T1]{fontenc}
\usepackage{mathtools}
\usepackage{ytableau}

\hypersetup{
	colorlinks=true, %set true if you want colored links
	linktoc=all, %set to all if you want both sections and subsections linked
	linkcolor=blue} %choose some color if you want links to stand out

\numberwithin{equation}{section}

\theoremstyle{theorem}
\newtheorem{theorem}{Theorem}[section]
\newtheorem*{theorem*}{Theorem}

\newtheorem{lemma}[theorem]{Lemma}

\newtheorem{innercustomgeneric}{\customgenericname}
\providecommand{\customgenericname}{}
\newcommand{\newcustomtheorem}[2]{%
	\newenvironment{#1}[1]
	{%
		\renewcommand\customgenericname{#2}%
		\renewcommand\theinnercustomgeneric{##1}%
		\innercustomgeneric
	}
	{\endinnercustomgeneric}
}
\newcustomtheorem{ctheorem}{Theorem}
\newcustomtheorem{clemma}{Lemma}

\theoremstyle{definition}
\newtheorem{definition}[theorem]{Definition}

\newtheorem*{example*}{Example}
\newtheorem*{examples*}{Examples}

\newtheorem*{remark*}{Remark}
\newtheorem*{remarks*}{Remarks}
\newtheorem*{note*}{Note}

\newtheoremstyle{named}{}{}{\itshape}{}{\bfseries}{.}{.5em}{\thmnote{#3} #1}
\theoremstyle{named}

\newtheoremstyle{customized}{}{}{\itshape}{}{\bfseries}{.}{.5em}{\thmnote{#3}}
\theoremstyle{customized}

\DeclareMathAlphabet{\mydutchcal}{U}{dutchcal}{m}{n}
\newcommand{\ocA}{\operatorname{\mydutchcal{A}}}

\newcommand{\ocR}{\operatorname{\mydutchcal{R}}}
\newcommand{\ocP}{\operatorname{\mydutchcal{P}}}
\newcommand{\ocQ}{\operatorname{\mydutchcal{Q}}}
\newcommand{\ocS}{\operatorname{\mydutchcal{S}}}
\newcommand{\ocTAG}{\operatorname{\mydutchcal{T}_{\mathrm{AG}}}}
\newcommand{\ocTR}{\operatorname{\mydutchcal{T}_{\mathrm{R}}}}

\newcommand{\qbinom}[2]{{\genfrac{[}{]}{0pt}{}{#1}{#2}}}

\newcommand{\RHS}{\operatorname{RHS}}

\newcommand{\GF}{\operatorname{GF}}

\newcommand{\len}{\mathsf{len}}

\newcommand{\mI}{\mathcal{I}}
\newcommand{\cL}{\mathcal{L}}

\newcommand{\bA}{\mathbf{A}}
\newcommand{\bW}{\mathbf{W}}

\newcommand{\sfR}{\mathsf{R}}
\newcommand{\sfG}{\mathsf{G}}
\newcommand{\sfB}{\mathsf{B}}

\newcommand{\ubA}{\mathbf{\underline{A}}}
\newcommand{\uba}{\boldsymbol{\underline{\alpha}}}
\newcommand{\ubb}{\boldsymbol{\underline{\beta}}}
\newcommand{\ubone}{\boldsymbol{\underline{1}}}
\newcommand{\ubtwo}{\boldsymbol{\underline{2}}}

\newcommand{\Mat}{\operatorname{Mat}}
\newcommand{\diag}{\operatorname{diag}}

\title[Russell's three-colored partitions]{Linked partition ideals and Russell's three-colored partitions}

\author[S. Chern]{Shane Chern}
\address{Fakult\"at f\"ur Mathematik, Universit\"at Wien, Oskar-Morgenstern-Platz 1, Wien 1090, Austria}
\email{chenxiaohang92@gmail.com, xiaohangc92@univie.ac.at}

\date{}



\begin{document}
	
\sloppy

\maketitle

\begin{abstract}
	Russell recently introduced a family of three-colored partitions in analogy with earlier work by Alladi and Gordon. He showed that their enumerations are surprisingly connected to the classic Rogers--Ramanujan identities. Using the method of linked partition ideals, we elaborate on Russell's results by further counting each part color. Unlike Russell's proofs, our arguments do not require any use of computer algebra systems. With these multivariate generating function relations, we are led to confirm a conjecture of Russell. Furthermore, our results offer a combinatorial understanding of a new triple Rogers--Ramanujan type identity.
\end{abstract}

\bigskip\bigskip

\begin{flushright}
	{\noindent\footnotesize\itshape Dedicated with admiration to Krishnaswami Alladi,\\ who has been bringing partition identities to us in a colorful way.}
\end{flushright}

\bigskip\bigskip

\section{Introduction}

A \emph{partition} is a weakly increasing sequence of positive integers known as \emph{parts}. In this work, we focus on partitions wherein each part is assigned with one of the three colors: \emph{red}, \emph{green}, and \emph{blue}, labeled by subscripts ``$\sfR$'', ``$\sfG$'', and ``$\sfB$''. Given such a three-colored partition $\lambda$, its \emph{size} $|\lambda|$ is the sum of all the parts. In addition, we denote by $\len(\lambda)$ the total number of parts in $\lambda$, and by $\len_\sfR(\lambda)$, $\len_\sfG(\lambda)$, and $\len_\sfB(\lambda)$ the number of red, green, and blue parts, respectively.

One of the starting points of the investigation on three-colored partitions under certain \emph{gap conditions} is the seminal work by Alladi and Gordon~\cite{AG1993,AG1995}. Let $\Omega$ be the set of three-colored partitions such that
\begin{itemize}[leftmargin=*,align=left,itemsep=1pt]
	\item each part size, counting all three colors, appears at most once;
	
	\item the following subpartitions are forbidden for all $j$:
	\begin{align*}
		\begin{array}{llll}
			j_\sfR+(j+1)_\sfR, & j_\sfG+(j+1)_\sfR, & j_\sfB+(j+1)_\sfR, & j_\sfB+(j+1)_\sfG.
		\end{array}
	\end{align*}
\end{itemize}
For $\mathrm{X}$ a partition set and $S$ a collection of parts, let $\mathrm{X}_S$ be the subset of $\mathrm{X}$ consisting of partitions having no parts from $S$. Alladi and Gordon~\cite[pp.~119--120, Lemma~1]{AG1993} proved that
\begin{align}\label{eq:AG-gf}
	&\sum_{\omega\in \Omega_{\{1_\sfR\}}} y_1^{\len_\sfR(\omega)} y_2^{\len_\sfG(\omega)} y_3^{\len_\sfB(\omega)} q^{|\omega|}\notag\\
	&\qquad = \sum_{n_1,n_2,n_3\ge 0} \frac{y_1^{n_1}y_2^{n_2}y_3^{n_3}q^{2\binom{n_1}{2}+\binom{n_2}{2}+\binom{n_3}{2}+n_1n_2+n_2n_3+n_3n_1+2n_1+n_2+n_3}}{(q;q)_{n_1}(q;q)_{n_2}(q;q)_{n_3}},
\end{align}
where the \emph{$q$-Pochhammer symbols} are defined for $n\in\mathbb{N}\cup\{\infty\}$,
\begin{align*}
	(A;q)_n:=\prod_{k=0}^{n-1} (1-A q^k).
\end{align*}

More recently, Russell~\cite{Rus2026} considered another set $\Lambda$ of three-colored partitions constrained by the conditions:
\begin{itemize}[leftmargin=*,align=left,itemsep=1pt]
	\item each part size, counting all three colors, appears at most once;
	
	\item the following subpartitions are forbidden for all $j$:
	\begin{align*}
		\begin{array}{llll}
			j_\sfR+(j+1)_\sfR, & j_\sfR+(j+2)_\sfR, & j_\sfR+(j+1)_\sfB, & j_\sfG+(j+1)_\sfR,\\
			j_\sfG+(j+2)_\sfR, & j_\sfG+(j+1)_\sfB, & j_\sfB+(j+1)_\sfR, & j_\sfB+(j+1)_\sfG.
		\end{array}
	\end{align*}
\end{itemize}
One of the results showed by Russell is \cite[p.~11, eq.~(3.10)]{Rus2026}:
\begin{align}\label{eq:Rus-gf}
	\sum_{\lambda\in \Lambda_{\{1_\sfR,2_\sfR,1_\sfB\}}} x^{\len(\lambda)} q^{|\lambda|} &= \sum_{n_1,n_2,n_3\ge 0} x^{n_1+n_2+n_3}q^{3\binom{n_1}{2}+\binom{n_2}{2}+\binom{n_3}{2}+n_1n_2+n_2n_3+n_3n_1}\notag\\
	&\quad\times \frac{q^{3n_1+n_2+2n_3}(-q;q)_{n_1+n_2+n_3}}{(q^2;q^2)_{n_1}(q^2;q^2)_{n_2}(q^2;q^2)_{n_3}}.
\end{align}
In particular, it exhibits a surprising connection with the classic \emph{Rogers--Ramanujan identities}~\cite{RR1919,Rog1894} according to \cite[p.~2, Theorem~2]{Rus2026}:
\begin{align*}
	\sum_{\lambda\in \Lambda_{\{1_\sfR,2_\sfR,1_\sfB\}}} q^{|\lambda|} = \frac{(-q;q)_\infty}{(q^2;q^5)_\infty (q^3;q^5)_\infty}.
\end{align*}

While Alladi and Gordon established \eqref{eq:AG-gf} by the method of weighted words and Russell established \eqref{eq:Rus-gf} using a certain $q$-difference equation, the enumeration problems for partitions under gap conditions fit well into the framework of \emph{linked partition ideals}, initially developed by Andrews~\cite{And1972} and recently revived in a series of projects mainly led by the author~\cite{ACL2022,Che2020,Che2022a,CL2020} after a semicentennial silence. In this work, we will elaborate on Russell's generating function identities by further counting each part color. More precisely, we have the following theorem.

\begin{theorem}\label{th:main-gf}
	Let
	\begin{align*}
		\GF_{\widehat{\Lambda}}(y_1,y_2,y_3) = \GF_{\widehat{\Lambda}}(y_1,y_2,y_3;q) := \sum_{\lambda\in \widehat{\Lambda}} y_1^{\len_\sfR(\lambda)} y_2^{\len_\sfG(\lambda)} y_3^{\len_\sfB(\lambda)} q^{|\lambda|}
	\end{align*}
	for $\widehat{\Lambda}$ a subset of $\Lambda$ and
	\begin{align*}
		\ocTR(y_1,y_2,y_3) = \ocTR(y_1,y_2,y_3;q) &:= \sum_{n_1,n_2,n_3\ge 0} q^{3\binom{n_1}{2}+\binom{n_2}{2}+\binom{n_3}{2}+n_1n_2+n_2n_3+n_3n_1}\notag\\
		&\;\quad\times \frac{y_1^{n_1}y_2^{n_2}y_3^{n_3}(-q;q)_{n_1+n_2+n_3}}{(q^2;q^2)_{n_1}(q^2;q^2)_{n_2}(q^2;q^2)_{n_3}}.
	\end{align*}
	Then
	\begin{align}
		\GF_{\Lambda}(y_1,y_2,y_3) &= \ocTR(y_1q,y_2q,y_3q^2) + y_3q \ocTR(y_1q^3,y_2q^3,y_3q^2),\label{eq:main-gf-1}\\
		\GF_{\Lambda_{\{1_{\sfR}\}}}(y_1,y_2,y_3) &= \ocTR(y_1q^2,y_2q^2,y_3q) + y_2q \ocTR(y_1q^4,y_2q^2,y_3q^3),\label{eq:main-gf-2}\\
		\GF_{\Lambda_{\{1_\sfR,2_\sfR,1_\sfB\}}}(y_1,y_2,y_3) &= \ocTR(y_1q^3,y_2q,y_3q^2),\label{eq:main-gf-3}\\
		\GF_{\Lambda_{\{1_\sfR,1_\sfG\}}}(y_1,y_2,y_3) &= \ocTR(y_1q^2,y_2q^2,y_3q).\label{eq:main-gf-4}
	\end{align}
\end{theorem}

These generating function identities will be placed within the framework of linked partition ideals. For this purpose, we use Alladi and Gordon's three-colored partitions as a warm-up to digest the concept of this setting in Section~\ref{sec:AG}. Then we move on to Russell's partitions in Section~\ref{sec:R}. As one may have already seen, the triple sums in \eqref{eq:main-gf-1}--\eqref{eq:main-gf-4} differ from that in \eqref{eq:AG-gf} by the extra factor $(-q;q)_{n_1+n_2+n_3}$ in the numerator. Taking account of the impact of this factor, in Section~\ref{sec:q-sum} we work out recurrences for generic multiple $q$-summations of this form, from which we finalize the proof of Theorem~\ref{th:main-gf} in Section~\ref{sec:GF}. Unlike Russell's proofs, our arguments do not require any use of computer algebra systems.

Notably, the multivariate generating functions in Theorem~\ref{th:main-gf} have several interesting implications. First, we are led to prove a conjecture of Russell~\cite[p.~13, Conjecture~3]{Rus2026} in Section~\ref{sec:R-conj}.

\begin{theorem}[Russell's conjecture]\label{th:Russell-conj}
	We have
	\begin{align}\label{eq:Russell-conj}
		\GF_{\Lambda_{\{1_{\sfR}\}}}(y^2,y,y) = (1+yq) \ocTR(y^2q^4,yq^2,yq).
	\end{align}
\end{theorem}

Furthermore, we will establish the following Rogers--Ramanujan type identity in Section~\ref{sec:RR-id} --- although it can be proven by $q$-hypergeometric means, we show how \eqref{eq:main-gf-4} may offer a simple \emph{combinatorial} understanding of this identity.

\begin{theorem}\label{th:RR-1}
	We have
	\begin{align}\label{eq:RR-1}
		&(-y;q)_\infty\notag\\
		&\  = \sum_{n_1,n_2,n_3\ge 0} \frac{(-1)^{n_1}x^{n_1+n_2}y^{n_3}q^{3\binom{n_1}{2}+\binom{n_2}{2}+\binom{n_3}{2}+n_1n_2+n_2n_3+n_3n_1}(-q;q)_{n_1+n_2+n_3}}{(q^2;q^2)_{n_1}(q^2;q^2)_{n_2}(q^2;q^2)_{n_3}}.
	\end{align}
\end{theorem}

\section{Linked partition ideals: Alladi and Gordon's partitions, a warm-up}\label{sec:AG}

The framework of linked partition ideals was originally designed for ordinary partitions. Here we slightly adjust the definition in \cite[Chapter~8]{And1998} or \cite[p.~5, Definition~2.1]{ACL2022} to accommodate colored partitions; a more general setting was provided in \cite[Section~2]{Che2020}.

\begin{definition}
	Suppose we are given
	\begin{itemize}[leftmargin=*,align=left,itemsep=1pt]
		\renewcommand{\labelitemi}{\scriptsize$\blacktriangleright$}
		
		\item a finite \emph{base set} $\Pi=\{\pi_1,\pi_2,\ldots,\pi_K\}$ of three-colored partitions with $\pi_1=\varnothing$, the empty partition,
		
		\item a \textit{linking map}, $\cL:\Pi\to P(\Pi)$, the power set of $\Pi$, with especially, $\cL(\pi_1)=\cL(\varnothing)=\Pi$ and $\pi_1=\varnothing\in \cL(\pi_k)$ for every $k$ with $1\le k\le K$,
		
		\item and a positive integral \textit{modulus} $T$, which is no smaller than the largest part among all three-colored partitions in $\Pi$.
	\end{itemize}
	We say a \textit{span one linked partition ideal} $\mI=\mI(\langle\Pi,\cL\rangle,T)$ is the collection of all three-colored partitions of the form
	\begin{align}\label{eq:decomp}
		\lambda&=\phi^0(\lambda_0)\oplus \phi^T(\lambda_1)\oplus \cdots \oplus \phi^{NT}(\lambda_N)\oplus \phi^{(N+1)T}(\pi_1)\oplus \phi^{(N+2)T}(\pi_1)\oplus \cdots\notag\\
		&=\phi^0(\lambda_0)\oplus \phi^T(\lambda_1)\oplus \cdots \oplus \phi^{NT}(\lambda_N),
	\end{align}
	where $\lambda_i\in\cL(\lambda_{i-1})$ for each $i$ and $\lambda_N$ is not the empty partition; we also include in $\mI$ the empty partition, which corresponds to $\phi^{0}(\pi_1)\oplus \phi^{T}(\pi_1)\oplus \cdots$. Here $\phi^m(\mu)$ gives a three-colored partition by adding $m$ to each part of the three-colored partition $\mu$ with colors preserved. Meanwhile, for any two three-colored partitions $\mu$ and $\nu$, $\mu\oplus\nu$ gives a new three-colored partition by collecting all parts in $\mu$ and $\nu$.
\end{definition}

It is clear from the gap conditions for Alladi and Gordon's partitions that $\Omega$ is identical to the span one linked partition ideal $\mI=\mI(\langle\Pi,\cL\rangle,1)$ with
\begin{align*}
	\Pi = \{\pi_1,\pi_2,\pi_3,\pi_4\} = \{\varnothing,1_\sfR,1_\sfG,1_\sfB\},
\end{align*}
and
\begin{align*}
	\cL(\pi_i) = \begin{cases}
		\{\pi_1,\pi_2,\pi_3,\pi_4\}, & i=1,\\
		\{\pi_1,\pi_3,\pi_4\}, & i=2,3,\\
		\{\pi_1,\pi_4\}, & i=4.
	\end{cases}
\end{align*}
As an example, the three-colored partition $1_\sfR+2_\sfB+4_\sfG+5_\sfG+8_\sfR$ in $\Omega$ takes the following form in $\mI$:
\begin{align*}
	\phi^0(1_\sfR) \oplus \phi^1(1_\sfB) \oplus \phi^2(\varnothing) \oplus \phi^3(1_\sfG) \oplus \phi^4(1_\sfG) \oplus \phi^5(\varnothing) \oplus \phi^6(\varnothing) \oplus \phi^7(1_\sfR).
\end{align*}

With such a decomposition in \eqref{eq:decomp}, we define $G_k(x)$ as the generating function for three-colored partitions in $\Omega$ whose first decomposed block equals $\pi_k$; that is,
\begin{align*}
	G_k(x)=G_k(x,y_1,y_2,y_3,q):=\sum_{\substack{\lambda\in\Omega\\\lambda_0=\pi_k}} x^{\len(\lambda)} y_1^{\len_\sfR(\lambda)} y_2^{\len_\sfG(\lambda)} y_3^{\len_\sfB(\lambda)} q^{|\lambda|}.
\end{align*}
It is plain from our construction that
\begin{align*}
	G_k(x)=x^{\len(\pi_k)} y_1^{\len_\sfR(\pi_k)} y_2^{\len_\sfG(\pi_k)} y_3^{\len_\sfB(\pi_k)} q^{|\pi_k|} \sum_{j:\pi_j\in\cL(\pi_k)} G_j(xq),
\end{align*}
so that we have a $q$-difference system
\begin{equation*}
	\begin{pmatrix}
		G_1(x)\\
		G_2(x)\\
		G_3(x)\\
		G_4(x)
	\end{pmatrix}
	=
	\begin{pmatrix}
		1\\
		& xy_1q\\
		&& xy_2q\\
		&&& xy_3q
	\end{pmatrix}.\begin{pmatrix}
	1 & 1 & 1 & 1\\
	1 & 0 & 1 & 1\\
	1 & 0 & 1 & 1\\
	1 & 0 & 0 & 1
	\end{pmatrix}.
	\begin{pmatrix}
		G_1(xq)\\
		G_2(xq)\\
		G_3(xq)\\
		G_4(xq)
	\end{pmatrix}.
\end{equation*}
If we further introduce
\begin{align*}
	\begin{pmatrix}
		F_1(x)\\
		F_2(x)\\
		F_3(x)\\
		F_4(x)
	\end{pmatrix} = \begin{pmatrix}
	1 & 1 & 1 & 1\\
	1 & 0 & 1 & 1\\
	1 & 0 & 1 & 1\\
	1 & 0 & 0 & 1
	\end{pmatrix}.
	\begin{pmatrix}
	G_1(x)\\
	G_2(x)\\
	G_3(x)\\
	G_4(x)
	\end{pmatrix},
\end{align*}
then the $q$-difference system for the $F$'s is
\begin{equation}\label{eq:AG-F-q-diff}
	\begin{pmatrix}
		F_1(x)\\
		F_2(x)\\
		F_3(x)\\
		F_4(x)
	\end{pmatrix}
	=
	\begin{pmatrix}
		1 & 1 & 1 & 1\\
		1 & 0 & 1 & 1\\
		1 & 0 & 1 & 1\\
		1 & 0 & 0 & 1
	\end{pmatrix}.\begin{pmatrix}
		1\\
		& xy_1q\\
		&& xy_2q\\
		&&& xy_3q
	\end{pmatrix}.
	\begin{pmatrix}
		F_1(xq)\\
		F_2(xq)\\
		F_3(xq)\\
		F_4(xq)
	\end{pmatrix}.
\end{equation}

Now our key observation is that for $\widehat{\Omega}$ a subset of $\Omega$, if we define
\begin{align*}
	\GF_{\widehat{\Omega}}(x) = \GF_{\widehat{\Omega}}(x,y_1,y_2,y_3;q) := \sum_{\lambda\in \widehat{\Omega}} x^{\len(\lambda)} y_1^{\len_\sfR(\lambda)} y_2^{\len_\sfG(\lambda)} y_3^{\len_\sfB(\lambda)} q^{|\lambda|},
\end{align*}
then
\begin{align*}
	\GF_{\Omega}(x) &= \sum_{k\in\{1,2,3,4\}} G_k(x) = F_1(x),\\
	\GF_{\Omega_{\{1_\sfR\}}}(x) &= \sum_{k\in\{1,3,4\}} G_k(x) = F_2(x) = F_3(x),\\
	\GF_{\Omega_{\{1_\sfR,1_\sfG\}}}(x) &= \sum_{k\in\{1,4\}} G_k(x) = F_4(x).
\end{align*}
Hence, it remains to find the \emph{unique} power series solution, as asserted by \cite[p.~12, Proposition~15]{Che2020}, to the system \eqref{eq:AG-F-q-diff} under the boundary condition $F_1(0) = F_2(0) = F_3(0) = F_4(0) = 1$.

\begin{theorem}\label{th:AG-gf}
	Let
	\begin{align*}
		\ocTAG(x,y_1,y_2,y_3;q) &:= \sum_{n_1,n_2,n_3\ge 0} q^{2\binom{n_1}{2}+\binom{n_2}{2}+\binom{n_3}{2}+n_1n_2+n_2n_3+n_3n_1}\notag\\
		&\;\quad\times \frac{x^{n_1+n_2+n_3}y_1^{n_1}y_2^{n_2}y_3^{n_3}}{(q;q)_{n_1}(q;q)_{n_2}(q;q)_{n_3}}.
	\end{align*}
	Then
	\begin{align}
		\GF_{\Omega}(x) &= \ocTAG(x,y_1q,y_2q,y_3q;q),\label{eq:AG-gf-1}\\
		\GF_{\Omega_{\{1_{\sfR}\}}}(x) &= \ocTAG(x,y_1q^2,y_2q,y_3q;q),\label{eq:AG-gf-2}\\
		\GF_{\Omega_{\{1_\sfR,1_\sfG\}}}(x) &= \ocTAG(x,y_1q^2,y_2q^2,y_3q;q).\label{eq:AG-gf-3}
	\end{align}
\end{theorem}

To establish these relations, we start with a family of generic multiple $q$-summations. Let $R$ be a fixed positive integer. We fix a symmetric matrix $\uba=(\alpha_{i,j})\in\Mat_{R\times R}(\mathbb{Z})$ and a vector $\ubA=(A_i)\in \mathbb{Z}_{>0}^R$. We also fix $J$ vectors $\boldsymbol{\underline{\gamma}^{(j)}}=(\gamma_{i}^{(j)})\in \mathbb{Z}^R$ for $j=1,2,\ldots,J$. Let $x_1,x_2,\ldots,x_J$ and $q$ be indeterminates and define the following formal multiple $q$-summation $\ocS(\ubb)=\ocS(\beta_1,\ldots,\beta_R)$ for $\ubb\in\mathbb{Z}^R$:
\begin{align}\label{eq:S-beta-sum-ref}
	\ocS(\ubb)&:=\sum_{n_1,\ldots,n_R\ge 0}q^{\sum_{i=1}^R \alpha_{i,i}\binom{n_i}{2}+\sum_{1\le i< j\le R}\alpha_{i,j}n_i n_j+ \sum_{i=1}^R \beta_i n_i}\notag\\
	&\;\quad\times \frac{x_1^{\sum_{i=1}^R \gamma_{i}^{(1)} n_i}\cdots x_J^{\sum_{i=1}^R \gamma_{i}^{(J)} n_i}}{(q^{A_1};q^{A_1})_{n_1}\cdots (q^{A_R};q^{A_R})_{n_R}}.
\end{align}
The key ingredient we need is \cite[p.~2015, Lemma~2.1]{Che2022a}:
\begin{align}\label{eq:S-rec}
	\ocS(\ubb)=\ocS(\ubb+\ubA_{\delta_r})+x_1^{\gamma_{r}^{(1)}}\cdots x_J^{\gamma_{r}^{(J)}}q^{\beta_r}\ocS(\ubb+\uba_r),
\end{align}
which is valid for every $r$ with $1\le r\le R$. Here we introduce the auxiliary vectors
\begin{align}\label{eq:vec-1}
	\uba_r:=(\alpha_{r,1},\ldots,\alpha_{r,R}), \qquad\qquad \ubA_{\delta_r} := (\delta_{r,1} A_1,\ldots, \delta_{r,R} A_R),
\end{align}
where the Kronecker delta $\delta_r$ is given by
\begin{align*}
	\delta_{r,i} := \begin{cases}
		1, & \text{if $i=r$},\\
		0, & \text{if $i\ne r$}.
	\end{cases}
\end{align*}

\begin{proof}[Proof of Theorem~\ref{th:AG-gf}]
	We specify $\ocS(\ubb)$ in \eqref{eq:S-beta-sum-ref} as
	\begin{align*}
		\ocA(\beta_1,\beta_2,\beta_3) &:= \sum_{n_1,n_2,n_3\ge 0} q^{2\binom{n_1}{2}+\binom{n_2}{2}+\binom{n_3}{2}+n_1n_2+n_2n_3+n_3n_1}\notag\\
		&\;\quad\times \frac{x^{n_1+n_2+n_3}y_1^{n_1}y_2^{n_2}y_3^{n_3}q^{\beta_1n_1+\beta_2n_2+\beta_3n_3}}{(q;q)_{n_1}(q;q)_{n_2}(q;q)_{n_3}}.
	\end{align*}
	Note that after the dilation $x\mapsto xq$, the sum $\ocA(\beta_1,\beta_2,\beta_3)$ becomes $\ocA(\beta_1+1,\beta_2+1,\beta_3+1)$. It is sufficient to show that
	\begin{equation*}
		\begin{pmatrix}
			\ocA(1,1,1)\\
			\ocA(2,1,1)\\
			\ocA(2,1,1)\\
			\ocA(2,2,1)
		\end{pmatrix}
		=
		\begin{pmatrix}
			1 & 1 & 1 & 1\\
			1 & 0 & 1 & 1\\
			1 & 0 & 1 & 1\\
			1 & 0 & 0 & 1
		\end{pmatrix}.\begin{pmatrix}
			1\\
			& xy_1q\\
			&& xy_2q\\
			&&& xy_3q
		\end{pmatrix}.
		\begin{pmatrix}
			\ocA(2,2,2)\\
			\ocA(3,2,2)\\
			\ocA(3,2,2)\\
			\ocA(3,3,2)
		\end{pmatrix}.
	\end{equation*}
	In other words, there are three relations to establish:
	\begin{align*}
		\ocA(1,1,1) &= \ocA(2,2,2) + xy_1q \ocA(3,2,2) + xy_2q \ocA(3,2,2) + xy_3q \ocA(3,3,2),\\
		\ocA(2,1,1) &= \ocA(2,2,2) + xy_2q \ocA(3,2,2) + xy_3q \ocA(3,3,2),\\
		\ocA(2,2,1) &= \ocA(2,2,2) + xy_3q \ocA(3,3,2).
	\end{align*}
	For the third relation, we apply \eqref{eq:S-rec} to $\ocA(2,2,1)$ at $r=3$; for the second relation, we apply \eqref{eq:S-rec} to $\ocA(2,1,1)$ at $r=2$ and then use the last relation; for the first relation, we apply \eqref{eq:S-rec} to $\ocA(1,1,1)$ at $r=1$ and then use the second relation. By doing so, the proof is finished.
\end{proof}

\section{Linked partition ideals: Russell's partitions}\label{sec:R}

For the moment, we move on to Russell's partitions. Once again from the gap conditions, it is plain that $\Lambda$ is identical to the span one linked partition ideal $\mI=\mI(\langle\Pi,\cL\rangle,2)$ with
\begin{align*}
	\Pi &= \{\pi_1,\pi_2,\pi_3,\pi_4,\pi_5,\pi_6,\pi_7,\pi_8,\pi_9,\pi_{10}\}\\
	& = \{\varnothing,1_\sfR,1_\sfG,1_\sfB,1_\sfR+2_\sfG,1_\sfG+2_\sfG,1_\sfB+2_\sfB,2_\sfR,2_\sfG,2_\sfB\},
\end{align*}
and
\begin{align*}
	\cL(\pi_i) = \begin{cases}
		\{\pi_1,\pi_2,\pi_3,\pi_4,\pi_5,\pi_6,\pi_7,\pi_8,\pi_9,\pi_{10}\}, & i=1,4,\\
		\{\pi_1,\pi_3,\pi_4,\pi_6,\pi_7,\pi_8,\pi_9,\pi_{10}\}, & i=2,3,\\
		\{\pi_1,\pi_3,\pi_6,\pi_9,\pi_{10}\}, & i=5,6,8,9,\\
		\{\pi_1,\pi_4,\pi_7,\pi_8,\pi_9,\pi_{10}\}, & i=7,10.
	\end{cases}
\end{align*}

Now define
\begin{align*}
	G_k(x)=G_k(x,y_1,y_2,y_3,q):=\sum_{\substack{\lambda\in\Omega\\\lambda_0=\pi_k}} x^{\len(\lambda)} y_1^{\len_\sfR(\lambda)} y_2^{\len_\sfG(\lambda)} y_3^{\len_\sfB(\lambda)} q^{|\lambda|}.
\end{align*}
Then
\begin{align*}
	G_k(x)=x^{\len(\pi_k)} y_1^{\len_\sfR(\pi_k)} y_2^{\len_\sfG(\pi_k)} y_3^{\len_\sfB(\pi_k)} q^{|\pi_k|} \sum_{j:\pi_j\in\cL(\pi_k)} G_j(xq^2),
\end{align*}
so that this time we have the $q$-difference system
\begin{equation}
	\begin{pmatrix}
		G_1(x)\\
		G_2(x)\\
		\vdots\\
		G_{10}(x)
	\end{pmatrix}
	=
	\bW.\bA.
	\begin{pmatrix}
		G_1(xq^2)\\
		G_2(xq^2)\\
		\vdots\\
		G_{10}(xq^2)
	\end{pmatrix},
\end{equation}
where
\begin{align*}
	\bW=\diag(1,xy_1q,xy_2q,xy_3q,x^2y_1y_2q^3,x^2y_2^2q^3,x^2y_3^2q^3,xy_1q^2,xy_2q^2,xy_3q^2),
\end{align*}
and
\begin{align*}
	\bA=\begin{pmatrix}
		1 & 1 & 1 & 1 & 1 & 1 & 1 & 1 & 1 & 1\\
		1 & 0 & 1 & 1 & 0 & 1 & 1 & 1 & 1 & 1\\
		1 & 0 & 1 & 1 & 0 & 1 & 1 & 1 & 1 & 1\\
		1 & 1 & 1 & 1 & 1 & 1 & 1 & 1 & 1 & 1\\
		1 & 0 & 1 & 0 & 0 & 1 & 0 & 0 & 1 & 1\\
		1 & 0 & 1 & 0 & 0 & 1 & 0 & 0 & 1 & 1\\
		1 & 0 & 0 & 1 & 0 & 0 & 1 & 1 & 1 & 1\\
		1 & 0 & 1 & 0 & 0 & 1 & 0 & 0 & 1 & 1\\
		1 & 0 & 1 & 0 & 0 & 1 & 0 & 0 & 1 & 1\\
		1 & 0 & 0 & 1 & 0 & 0 & 1 & 1 & 1 & 1
	\end{pmatrix}.
\end{align*}
Let us write
\begin{align*}
	\begin{pmatrix}
		F_1(x)\\
		F_2(x)\\
		\vdots\\
		F_{10}(x)
	\end{pmatrix}=\bA.\begin{pmatrix}
		G_1(x)\\
		G_2(x)\\
		\vdots\\
		G_{10}(x)
	\end{pmatrix}.
\end{align*}
Then
\begin{align}\label{eq:F}
	\begin{pmatrix}
		F_1(x)\\
		F_2(x)\\
		\vdots\\
		F_{10}(x)
	\end{pmatrix}=\bA.\bW.\begin{pmatrix}
		F_1(xq^2)\\
		F_2(xq^2)\\
		\vdots\\
		F_{10}(xq^2)
	\end{pmatrix}.
\end{align}

Finally, by abuse of notation, we define
\begin{align*}
	\GF_{\widehat{\Lambda}}(x) = \GF_{\widehat{\Lambda}}(x,y_1,y_2,y_3;q) := \sum_{\lambda\in \widehat{\Lambda}} x^{\len(\lambda)} y_1^{\len_\sfR(\lambda)} y_2^{\len_\sfG(\lambda)} y_3^{\len_\sfB(\lambda)} q^{|\lambda|}
\end{align*}
for $\widehat{\Lambda}$ a subset of $\Lambda$. Then it is true that
\begin{align*}
	\GF_{\Lambda}(x) &= \sum_{k\in\{1,2,3,4,5,6,7,8,9,10\}} G_k(x) = F_1(x)=F_4(x),\\
	\GF_{\Lambda_{\{1_\sfR\}}}(x) &= \sum_{k\in\{1,3,4,6,7,8,9,10\}} G_k(x) = F_2(x)=F_3(x),\\
	\GF_{\Lambda_{\{1_\sfR,2_\sfR,1_\sfB\}}}(x) &= \sum_{k\in\{1,3,6,9,10\}} G_k(x) = F_5(x)=F_6(x)=F_8(x)=F_9(x),\\
	\GF_{\Lambda_{\{1_\sfR,1_\sfG\}}}(x) &= \sum_{k\in\{1,4,7,8,9,10\}} G_k(x) = F_7(x)=F_{10}(x).
\end{align*}
These relations produce the boundary condition $F_1(0) = \cdots = F_{10}(0) = 1$ for the $q$-difference system \eqref{eq:F}.

\section{Multiple $q$-summations and their recurrences}\label{sec:q-sum}

As we have remarked, the triple sums in Theorem~\ref{th:main-gf} contain the factor $(-q;q)_{n_1+n_2+n_3}$ in the numerator, making the recurrences in \eqref{eq:S-rec} no longer practicable. In this section, we resolve this issue in a more general setting.

Using the same assumptions as those for \eqref{eq:S-beta-sum-ref}, we define
\begin{align}\label{eq:P-beta-sum-ref}
	\ocP(\ubb)&:=\sum_{n_1,\ldots,n_R\ge 0}q^{\sum_{i=1}^R \alpha_{i,i}\binom{n_i}{2}+\sum_{1\le i< j\le R}\alpha_{i,j}n_i n_j+ \sum_{i=1}^R \beta_i n_i}\notag\\
	&\;\quad\times \frac{x_1^{\sum_{i=1}^R \gamma_{i}^{(1)} n_i}\cdots x_J^{\sum_{i=1}^R \gamma_{i}^{(J)} n_i}(-q;q)_{n_1+\cdots+n_R}}{(q^{A_1};q^{A_1})_{n_1}\cdots (q^{A_R};q^{A_R})_{n_R}}.
\end{align}
Besides the auxiliary vectors in \eqref{eq:vec-1}, we further introduce
\begin{align}
	\ubone := (1,\ldots,1), \qquad\qquad \ubtwo_{\sigma,\chi_{<r}} := (2\chi_{\sigma^{-1}(1)<r},\ldots,2\chi_{\sigma^{-1}(R)<r}),
\end{align}
where the indicator function $\chi_{<r}$ is given by
\begin{align*}
	\chi_{i< r} := \begin{cases}
		1, & \text{if $i<r$},\\
		0, & \text{if $i\ge r$},
	\end{cases}
\end{align*}
and $\sigma$ is a permutation of $\{1,\ldots,R\}$.

We first establish an analog to \eqref{eq:S-rec}.

\begin{lemma}\label{le:rec-for-P}
	For $1\le r\le R$, we have
	\begin{align}\label{eq:rec-for-P}
		\ocP(\ubb)&=\ocP(\ubb+\ubA_{\delta_r})\notag\\
		&\quad+x_1^{\gamma_{r}^{(1)}}\cdots x_J^{\gamma_{r}^{(J)}}q^{\beta_r}\ocP(\ubb+\uba_r)\notag\\
		&\quad+x_1^{\gamma_{r}^{(1)}}\cdots x_J^{\gamma_{r}^{(J)}}q^{\beta_r+1}\ocP(\ubb+\uba_r+\ubone).
	\end{align}
\end{lemma}

\begin{proof}
	Recall that $\uba$ is symmetric so that $\alpha_{i,j}=\alpha_{j,i}$ for $1\le i,j\le R$. We have
	\begin{align*}
		&\ocP(\ubb) - \ocP(\ubb+\ubA_{\delta_r})\\
		&\qquad=\sum_{n_1,\ldots,n_R\ge 0}\frac{q^{\sum_i \alpha_{i,i}\binom{n_i}{2}+\sum_{i<j}\alpha_{i,j}n_i n_j}(1-q^{n_rA_r})(-q;q)_{n_1+\cdots+n_R}}{(q^{A_1};q^{A_1})_{n_1}\cdots(q^{A_r};q^{A_r})_{n_r}\cdots (q^{A_R};q^{A_R})_{n_R}}\\
		&\qquad\quad\times x_1^{\sum_i \gamma_{1,i} n_i}\cdots x_J^{\sum_i \gamma_{J,i} n_i} q^{\sum_i \beta_i n_i}\\
		&\qquad=\sum_{n_1,\ldots,n_R\ge 0}\frac{q^{\sum_i \alpha_{i,i}\binom{n_i}{2}+\sum_{i<j}\alpha_{i,j}n_i n_j}(1+q^{n_1+\cdots+n_R})(-q;q)_{n_1+\cdots+n_R-1}}{(q^{A_1};q^{A_1})_{n_1}\cdots(q^{A_r};q^{A_r})_{n_r-1}\cdots (q^{A_R};q^{A_R})_{n_R}}\\
		&\qquad\quad\times x_1^{\sum_i \gamma_{1,i} n_i}\cdots x_J^{\sum_i \gamma_{J,i} n_i} q^{\sum_i \beta_i n_i}\\
		&\qquad=x_1^{\gamma_{1,r}}\cdots x_J^{\gamma_{J,r}}q^{\beta_r}\sum_{n_1,\ldots,n_R\ge 0}\frac{q^{\sum_i \alpha_{i,i}\binom{n_i}{2}+\sum_{i<j}\alpha_{i,j}n_i n_j}(-q;q)_{n_1+\cdots+n_R}}{(q^{A_1};q^{A_1})_{n_1}\cdots(q^{A_r};q^{A_r})_{n_r}\cdots (q^{A_R};q^{A_R})_{n_R}}\\
		&\qquad\quad\times x_1^{\sum_i \gamma_{1,i} n_i}\cdots x_J^{\sum_i \gamma_{J,i} n_i} \big(q^{\sum_i (\beta_i+\alpha_{r,i}) n_i}+q^{1 + \sum_i (\beta_i+\alpha_{r,i}+1) n_i}\big),
	\end{align*}
	thereby yielding the desired identity.
\end{proof}

Next, we specialize $\ocP(\ubb)$ in \eqref{eq:P-beta-sum-ref} with $A_1=\cdots=A_R=2$. That is, we introduce
\begin{align}\label{eq:Q-beta-sum-ref}
	\ocQ(\ubb)&:=\sum_{n_1,\ldots,n_R\ge 0}q^{\sum_{i=1}^R \alpha_{i,i}\binom{n_i}{2}+\sum_{1\le i< j\le R}\alpha_{i,j}n_i n_j+ \sum_{i=1}^R \beta_i n_i}\notag\\
	&\;\quad\times \frac{x_1^{\sum_{i=1}^R \gamma_{1,i} n_i}\cdots x_J^{\sum_{i=1}^R \gamma_{J,i} n_i}(-q;q)_{n_1+\cdots+n_R}}{(q^{2};q^{2})_{n_1}\cdots (q^{2};q^{2})_{n_R}}.
\end{align}

We prove the following identities that can be viewed as a generalization of Russell's atomic relations in \cite[Section~2.1]{Rus2026}.

\begin{lemma}\label{le:rec-for-Q}
	For $\sigma$ a permutation of $\{1,\ldots,R\}$, we have
	\begin{align}\label{eq:rec-for-Q}
		\ocQ(\ubb)&=\ocQ(\ubb+\ubone)\notag\\
		&\quad+x_1^{\gamma_{\sigma(1)}^{(1)}}\cdots x_J^{\gamma_{\sigma(1)}^{(J)}}q^{\beta_{\sigma(1)}}\ocQ(\ubb+\uba_{\sigma(1)}+\ubtwo_{\sigma,\chi_{<1}})\notag\\
		&\quad+x_1^{\gamma_{\sigma(2)}^{(1)}}\cdots x_J^{\gamma_{\sigma(2)}^{(J)}}q^{\beta_{\sigma(2)}}\ocQ(\ubb+\uba_{\sigma(2)}+\ubtwo_{\sigma,\chi_{<2}})\notag\\
		&\quad+\cdots\notag\\
		&\quad+x_1^{\gamma_{\sigma(R)}^{(1)}}\cdots x_J^{\gamma_{\sigma(R)}^{(J)}}q^{\beta_{\sigma(R)}}\ocQ(\ubb+\uba_{\sigma(R)}+\ubtwo_{\sigma,\chi_{<R}}).
	\end{align}
\end{lemma}

\begin{proof}
	We have
	\begin{align*}
		&\ocQ(\ubb) - \ocQ(\ubb+\ubone)\\
		&\qquad=\sum_{n_1,\ldots,n_R\ge 0}\frac{q^{\sum_i \alpha_{i,i}\binom{n_i}{2}+\sum_{i<j}\alpha_{i,j}n_i n_j}(1-q^{n_1+\cdots+n_R})(-q;q)_{n_1+\cdots+n_R}}{(q^{2};q^{2})_{n_1}\cdots (q^{2};q^{2})_{n_R}}\\
		&\qquad\quad\times x_1^{\sum_i \gamma_{1,i} n_i}\cdots x_J^{\sum_i \gamma_{J,i} n_i} q^{\sum_i \beta_i n_i}\\
		&\qquad=\sum_{n_1,\ldots,n_R\ge 0}\frac{q^{\sum_i \alpha_{i,i}\binom{n_i}{2}+\sum_{i<j}\alpha_{i,j}n_i n_j}(1-q^{2(n_1+\cdots+n_R)})(-q;q)_{n_1+\cdots+n_R-1}}{(q^{2};q^{2})_{n_1}\cdots (q^{2};q^{2})_{n_R}}\\
		&\qquad\quad\times x_1^{\sum_i \gamma_{1,i} n_i}\cdots x_J^{\sum_i \gamma_{J,i} n_i} q^{\sum_i \beta_i n_i}.
	\end{align*}
	Now write
	\begin{align*}
		1-q^{2(n_1+\cdots+n_R)} &= (1-q^{2\sigma(1)}) + q^{2\sigma(1)}(1-q^{2\sigma(2)}) + q^{2(\sigma(1)+\sigma(2))}(1-q^{2\sigma(3)})\\
		&\quad + q^{2(\sigma(1)+\sigma(2)+\cdots+\sigma(R-1))}(1-q^{2\sigma(R)}).
	\end{align*}
	The claimed relation then follows.
\end{proof}

\section{Generating functions}\label{sec:GF}

Now we are in a position to finish the proof of Theorem~\ref{th:main-gf}. Throughout, for the multiple $q$-summation $\ocQ(\ubb)$ in \eqref{eq:Q-beta-sum-ref}, we choose
\begin{align*}
	\uba=\begin{pmatrix}
		3 & 1 & 1\\1 & 1 & 1\\1 & 1 & 1
	\end{pmatrix},
\end{align*}
and
\begin{alignat*}{2}
	x_1&\mapsto x,\qquad\qquad &&\boldsymbol{\underline{\gamma}^{(1)}}=(1,1,1),\\
	x_2&\mapsto y_1,\qquad\qquad &&\boldsymbol{\underline{\gamma}^{(2)}}=(1,0,0),\\
	x_3&\mapsto y_2,\qquad\qquad &&\boldsymbol{\underline{\gamma}^{(3)}}=(0,1,0),\\
	x_4&\mapsto y_3,\qquad\qquad &&\boldsymbol{\underline{\gamma}^{(4)}}=(0,0,1).
\end{alignat*}
In other words, we focus on the triple sum
\begin{align}\label{eq:R-def}
	\ocR(\beta_1,\beta_2,\beta_3)&:=\sum_{n_1,n_2,n_3\ge 0} q^{3\binom{n_1}{2}+\binom{n_2}{2}+\binom{n_3}{2}+n_1n_2+n_2n_3+n_3n_1}\notag\\
	&\;\quad\times \frac{x^{n_1+n_2+n_3}y_1^{n_1}y_2^{n_2}y_3^{n_3}q^{\beta_1n_1+\beta_2n_2+\beta_3n_3}(-q;q)_{n_1+n_2+n_3}}{(q^2;q^2)_{n_1}(q^2;q^2)_{n_2}(q^2;q^2)_{n_3}}.
\end{align}

Note that after the dilation $x\mapsto xq^2$, the sum $\ocA(\beta_1,\beta_2,\beta_3)$ becomes $\ocA(\beta_1+2,\beta_2+2,\beta_3+2)$. According to the arguments in Section~\ref{sec:R}, Theorem~\ref{th:main-gf} follows as long as the following four relations are established:
\begin{align*}
	F_1(x) = F_4(x) &= \ocR(1,1,2) + xy_3q \ocR(3,3,2),\\
	F_2(x) = F_3(x) &= \ocR(2,2,1) + xy_2q \ocR(4,2,3),\\
	F_5(x) = F_6(x) = F_8(x) = F_9(x) &= \ocR(3,1,2),\\
	F_7(x) = F_{10}(x) &= \ocR(2,2,1).
\end{align*}
In light of \eqref{eq:F}, we are only left to prove the following lemmas.

\begin{lemma}
	We have
	\begin{align}\label{eq:312}
		\ocR(3,1,2) &= \ocR(3,3,4) + xy_3q^3 \ocR(5,5,4)\notag\\
		&\quad + xy_2q\ocR(4,4,3) + x^2y_2^2q^4 \ocR(6,4,5)\notag\\
		&\quad + x^2y_2^2q^3 \ocR(5,3,4) + xy_2q^2 \ocR(5,3,4)\notag\\
		&\quad + xy_3q^2 \ocR(4,4,3).
	\end{align}
\end{lemma}

\begin{proof}
	This is
	\begin{align*}
		F_5(x) &= F_1(xq^2) + xy_2qF_3(xq^2) + x^2y_2^2q^3F_6(xq^2)\\
		&\quad + xy_2q^2F_9(xq^2) + xy_3q^2F_{10}(xq^2).
	\end{align*}
	We first apply \eqref{eq:rec-for-P} at $r=2$ to get
	\begin{align*}
		\ocR(3,1,2) = \ocR(3,3,2) + xy_2q \ocR(4,2,3) + xy_2q^2 \ocR(5,3,4).
	\end{align*}
	For the terms $\ocR(3,3,2)$ and $\ocR(4,2,3)$ on the right-hand side of the above, we use \eqref{eq:rec-for-P} at $r=3$ and $2$, respectively. Then
	\begin{align*}
		\ocR(3,3,2) &= \ocR(3,3,4) + xy_3q^2 \ocR(4,4,3) + xy_3q^3 \ocR(5,5,4),\\
		\ocR(4,2,3) &= \ocR(4,4,3) + xy_2q^2 \ocR(5,3,4) + xy_2q^3 \ocR(6,4,5).
	\end{align*}
	Substituting them into the identity for $\ocR(3,1,2)$, we arrive at \eqref{eq:312}.
\end{proof}

\begin{lemma}
	We have
	\begin{align}\label{eq:112+332}
		\ocR(1,1,2) + xy_3q \ocR(3,3,2) &= \ocR(3,3,4) + xy_3q^3 \ocR(5,5,4)\notag\\
		&\quad + xy_1q\ocR(4,4,3) + x^2y_1y_2q^4 \ocR(6,4,5)\notag\\
		&\quad + xy_2q\ocR(4,4,3) + x^2y_2^2q^4 \ocR(6,4,5)\notag\\
		&\quad + xy_3q\ocR(3,3,4) + x^2y_3^2q^4 \ocR(5,5,4)\notag\\
		&\quad + x^2y_1y_2q^3\ocR(5,3,4) + x^2y_2^2q^3 \ocR(5,3,4)\notag\\
		&\quad + x^2y_3^2q^3 \ocR(4,4,3) + xy_1q^2 \ocR(5,3,4)\notag\\
		&\quad + xy_2q^2 \ocR(5,3,4) + xy_3q^2 \ocR(4,4,3).
	\end{align}
\end{lemma}

\begin{proof}
	This is
	\begin{align*}
		F_1(x) &= F_1(xq^2) + xy_1qF_2(xq^2) + xy_2qF_3(xq^2) + xy_3qF_4(xq^2)\\
		&\quad + x^2y_1y_2q^3F_5(xq^2) + x^2y_2^2q^3F_6(xq^2) + x^2y_3^2q^3F_7(xq^2)\\
		&\quad + xy_1q^2F_8(xq^2) + xy_2q^2F_9(xq^2) + xy_3q^2F_{10}(xq^2).
	\end{align*}
	We first derive by using \eqref{eq:rec-for-P} at $r=1$ that
	\begin{align*}
		\ocR(1,1,2) = \ocR(3,1,2) + xy_1q \ocR(4,2,3) + xy_1q^2 \ocR(5,3,4).
	\end{align*}
	For the term $\ocR(3,1,2)$, we invoke \eqref{eq:312}; for the term $\ocR(4,2,3)$, we need \eqref{eq:rec-for-P} at $r=2$,
	\begin{align*}
		\ocR(4,2,3) = \ocR(4,4,3) + xy_2q^2 \ocR(5,3,4) + xy_2q^3 \ocR(6,4,5).
	\end{align*}
	Next, we apply \eqref{eq:rec-for-P} at $r=3$,
	\begin{align*}
		\ocR(3,3,2) = \ocR(3,3,4) + xy_3q^2 \ocR(4,4,3) + xy_3q^3 \ocR(5,5,4).
	\end{align*}
	Therefore, \eqref{eq:112+332} is true.
\end{proof}

\begin{lemma}
	We have
	\begin{align}\label{eq:221}
		\ocR(2,2,1) &= \ocR(3,3,4) + xy_3q^3 \ocR(5,5,4)\notag\\
		&\quad + xy_3q\ocR(3,3,4) + x^2y_3^2q^4 \ocR(5,5,4)\notag\\
		&\quad + x^2y_3^2q^3 \ocR(4,4,3) + xy_1q^2 \ocR(5,3,4)\notag\\
		&\quad + xy_2q^2 \ocR(5,3,4) + xy_3q^2 \ocR(4,4,3).
	\end{align}
\end{lemma}

\begin{proof}
	This is
	\begin{align*}
		F_7(x) &= F_1(xq^2) + xy_3qF_4(xq^2) + x^2y_3^2q^3F_7(xq^2)\\
		&\quad + xy_1q^2F_8(xq^2) + xy_2q^2F_9(xq^2) + xy_3q^2F_{10}(xq^2).
	\end{align*}
	We start with \eqref{eq:rec-for-P} at $r=3$,
	\begin{align*}
		\ocR(2,2,1) = \ocR(2,2,3) + xy_3q \ocR(3,3,2) + xy_3q^2 \ocR(4,4,3).
	\end{align*}
	For the term $\ocR(3,3,2)$, applying \eqref{eq:rec-for-P} at $r=3$ once more,
	\begin{align*}
		\ocR(3,3,2) = \ocR(3,3,4) + xy_3q^2 \ocR(4,4,3) + xy_3q^3 \ocR(5,5,4).
	\end{align*}
	For the term $\ocR(2,2,3)$, we need \eqref{eq:rec-for-Q} with the permutation $\sigma=(1,2,3)$ to get
	\begin{align*}
		\ocR(2,2,3) = \ocR(3,3,4) + xy_1q^2 \ocR(5,3,4) + xy_2q^2 \ocR(5,3,4) + xy_3q^3 \ocR(5,5,4).
	\end{align*}
	Then \eqref{eq:221} is established.
\end{proof}

\begin{lemma}
	We have
	\begin{align}\label{eq:221+423}
		\ocR(2,2,1) + xy_2q \ocR(4,2,3) &= \ocR(3,3,4) + xy_3q^3 \ocR(5,5,4)\notag\\
		&\quad + xy_2q\ocR(4,4,3) + x^2y_2^2q^4 \ocR(6,4,5)\notag\\
		&\quad + xy_3q\ocR(3,3,4) + x^2y_3^2q^4 \ocR(5,5,4)\notag\\
		&\quad + x^2y_2^2q^3 \ocR(5,3,4) + x^2y_3^2q^3 \ocR(4,4,3)\notag\\
		&\quad + xy_1q^2 \ocR(5,3,4) + xy_2q^2 \ocR(5,3,4)\notag\\
		&\quad + xy_3q^2 \ocR(4,4,3).
	\end{align}
\end{lemma}

\begin{proof}
	This is
	\begin{align*}
		F_2(x) &= F_1(xq^2) + xy_2qF_3(xq^2) + xy_3qF_4(xq^2) + x^2y_2^2q^3F_6(xq^2)\\
		&\quad + x^2y_3^2q^3F_7(xq^2) + xy_1q^2F_8(xq^2) + xy_2q^2F_9(xq^2) + xy_3q^2F_{10}(xq^2).
	\end{align*}
	In light of \eqref{eq:221}, we only need the following relation to prove \eqref{eq:221+423}:
	\begin{align*}
		\ocR(4,2,3) = \ocR(4,4,3) + xy_2q^2 \ocR(5,3,4) + xy_2q^3 \ocR(6,4,5),
	\end{align*}
	which can be shown by using \eqref{eq:rec-for-P} at $r=2$.
\end{proof}

\section{Russell's conjecture}\label{sec:R-conj}

We prove Russell's conjecture stated in Theorem~\ref{th:Russell-conj}. This time we further specify the triple sum $\ocR(\beta_1,\beta_2,\beta_3)$ in \eqref{eq:R-def} as
	\begin{align*}
		\ocR^{\dagger}(\beta_1,\beta_2,\beta_3)&:=\sum_{n_1,n_2,n_3\ge 0} q^{3\binom{n_1}{2}+\binom{n_2}{2}+\binom{n_3}{2}+n_1n_2+n_2n_3+n_3n_1}\notag\\
		&\;\quad\times \frac{y^{2n_1+n_2+n_3}q^{\beta_1n_1+\beta_2n_2+\beta_3n_3}(-q;q)_{n_1+n_2+n_3}}{(q^2;q^2)_{n_1}(q^2;q^2)_{n_2}(q^2;q^2)_{n_3}}.
	\end{align*}
	Recalling \eqref{eq:main-gf-2}, we need to show
	\begin{align*}
		\ocR^{\dagger}(2,2,1) + yq \ocR^{\dagger}(4,2,3) = (1+yq) \ocR^{\dagger}(4,2,1).
	\end{align*}
	Applying \eqref{eq:rec-for-P} at $r=1$,
	\begin{align*}
		\ocR^{\dagger}(2,2,1) = \ocR^{\dagger}(4,2,1) + y^2 q^2 \ocR^{\dagger}(5,3,2) + y^2 q^3 \ocR^{\dagger}(6,4,3).
	\end{align*}
	Thus, it is sufficient to show that
	\begin{align*}
		\ocR^{\dagger}(4,2,1) = \ocR^{\dagger}(4,2,3) + y q \ocR^{\dagger}(5,3,2) + y q^2 \ocR^{\dagger}(6,4,3),
	\end{align*}
	which follows from \eqref{eq:rec-for-P} at $r=3$.

\section{Rogers--Ramanujan type identity}\label{sec:RR-id}

In this part, we establish the Rogers--Ramanujan type identity \eqref{eq:RR-1}. Here we offer two different standpoints, one combinatorial and the other $q$-hypergeometric, to understand this identity.

\subsection{A combinatorial view}\label{sec:RR-id-comb}

Let $\widetilde{\Lambda}$ be either $\Lambda$ or $\Lambda_{\{1_\sfR,1_\sfG\}}$. Our key observation is that given any $\lambda\in \widetilde{\Lambda}$, the red and green colors are interchangeable for the smallest part that is not blue. In doing so, the parity of the numbers of red parts in the two partitions differs. This fact immediately gives us the equidistribution that for every $N\ge 0$, $M\ge 0$, and $L\ge 1$,
\begin{align*}
	&\#\big\{\lambda\in \widetilde{\Lambda}(N): \text{$\len_\sfB(\lambda)=M$, $\len_\sfR(\lambda)+\len_\sfG(\lambda) = L$ and $\len_\sfR(\lambda)$ is even}\big\}\notag\\
	&\qquad = \#\big\{\lambda\in \widetilde{\Lambda}(N): \text{$\len_\sfB(\lambda)=M$, $\len_\sfR(\lambda)+\len_\sfG(\lambda) = L$ and $\len_\sfR(\lambda)$ is odd}\big\},
\end{align*}
where $\widetilde{\Lambda}(N)$ is the subset of partitions of size $N$ in $\widetilde{\Lambda}$. Therefore,
\begin{align*}
	\sum_{\lambda\in \widetilde{\Lambda}} (-1)^{\len_\sfR(\lambda)} x^{\len_\sfR(\lambda)+\len_\sfG(\lambda)} y^{\len_\sfB(\lambda)} q^{|\lambda|} = \sum_{\substack{\lambda\in \widetilde{\Lambda}\\\len_\sfR(\lambda)=\len_\sfG(\lambda)=0}} y^{\len_\sfB(\lambda)} q^{|\lambda|},
\end{align*}
while the latter is simply the generating function for partitions into distinct parts, thereby producing $(-yq;q)_\infty$. In other words, we have the following signed counting for Russell's three-colored partitions.

\begin{theorem}
	Let $\widetilde{\Lambda}$ be either $\Lambda$ or $\Lambda_{\{1_\sfR,1_\sfG\}}$. We have
	\begin{align}\label{eq:signed}
		 \sum_{\lambda\in \widetilde{\Lambda}} (-1)^{\len_\sfR(\lambda)} x^{\len_\sfR(\lambda)+\len_\sfG(\lambda)} y^{\len_\sfB(\lambda)} q^{|\lambda|} = (-yq;q)_\infty.
	\end{align}
\end{theorem}

In particular, if we choose $\widetilde{\Lambda} = \Lambda_{\{1_\sfR,1_\sfG\}}$, the left-hand side of \eqref{eq:signed} becomes
\begin{align*}
	&\sum_{\lambda\in \Lambda_{\{1_\sfR,1_\sfG\}}} (-1)^{\len_\sfR(\lambda)} x^{\len_\sfR(\lambda)+\len_\sfG(\lambda)} y^{\len_\sfB(\lambda)} q^{|\lambda|}\\
	&\qquad =\sum_{n_1,n_2,n_3\ge 0} q^{3\binom{n_1}{2}+\binom{n_2}{2}+\binom{n_3}{2}+n_1n_2+n_2n_3+n_3n_1+2n_1+2n_2+n_3}\notag\\
	&\qquad\quad\times \frac{(-1)^{n_1}x^{n_1+n_2}y^{n_3}(-q;q)_{n_1+n_2+n_3}}{(q^2;q^2)_{n_1}(q^2;q^2)_{n_2}(q^2;q^2)_{n_3}},
\end{align*}
where we have used the generating function identity \eqref{eq:main-gf-4}. Now making the dilation $(x,y)\mapsto (xq^{-2},yq^{-1})$ gives \eqref{eq:RR-1}. In this sense, we are led to a combinatorial interpretation of the desired Rogers--Ramanujan type identity.

\subsection{A $q$-hypergeometric view}\label{sec:RR-id-q}

Finally, we prove \eqref{eq:RR-1} by $q$-hypergeometric means. Let us start with the sums over $n_1$ and $n_2$ and rewrite the right-hand side of \eqref{eq:RR-1} as
\begin{align*}
	\sum_{n_3\ge 0} \frac{y^{n_3} q^{\binom{n_3}{2}}}{(q;q)_{n_3}} \sum_{n_1,n_2\ge 0} \frac{(-1)^{n_1}x^{n_1+n_2}q^{3\binom{n_1}{2}+\binom{n_2}{2}+n_1n_2+n_3(n_1+n_2)}(-q^{n_3+1};q)_{n_1+n_2}}{(q^2;q^2)_{n_1}(q^2;q^2)_{n_2}}.
\end{align*}
If we introduce the auxiliary index $N := n_1+n_2$, the inner sums over $n_1$ and $n_2$ become
\begin{align*}
	\sum_{N\ge 0}\frac{x^N q^{\binom{N}{2}+n_3N}(-q^{n_3+1};q)_N}{(q^2;q^2)_N} \sum_{n_1\ge 0} (-1)^{n_1} q^{2\binom{n_1}{2}} \qbinom{N}{n_1}_{q^2} = 1,
\end{align*}
where we have used the $q$-binomial coefficient identity in \cite[p.~36, eq.~(3.3.6)]{And1998}:
\begin{align*}
	(z;q)_N = \sum_{n\ge 0} (-z)^n q^{\binom{n}{2}} \qbinom{N}{n}_q,
\end{align*}
so that the sum over $n_1$ becomes $(1;q^2)_N$, vanishing unless $N=0$. We conclude that
\begin{align*}
	\RHS\eqref{eq:RR-1} = \sum_{n_3\ge 0} \frac{y^{n_3} q^{\binom{n_3}{2}}}{(q;q)_{n_3}} = (-y;q)_\infty,
\end{align*}
as claimed, where Euler's second sum~\cite[p.~354, eq.~(II.2)]{GR2004} has been applied.

\subsection*{Acknowledgements}

This work was supported by the Austrian Science Fund (No.~10.55776/F1002). 

\bibliographystyle{amsplain}

\end{document}